\documentclass[12pt]{amsart}
\usepackage[utf8]{inputenc}

\usepackage{tikz-cd}
\usepackage{verbatim} %
\usepackage{graphicx}
\usepackage{mathtools}
\usepackage{mathrsfs}
\usepackage{enumerate}
\usepackage[T2A]{fontenc}
\usepackage[utf8]{inputenc}
\usepackage{MnSymbol}

\usepackage{tikz-cd}

\DeclareMathAlphabet\mathbb{U}{msb}{m}{n}

\def\ZZ{\mathbb{Z}}

\numberwithin{equation}{section}

\newtheorem{Theorem}{Theorem}[section]
\newtheorem{Corollary}[Theorem]{Corollary}
\newtheorem{Lemma}[Theorem]{Lemma}
\newtheorem{Proposition}[Theorem]{Proposition}
\theoremstyle{definition}

\newtheorem{Remark}[Theorem]{Remark}

\def\gg{\mathfrak{g}}
\def\FF{\mathcal{F}}
\def\GG{\mathcal{G}}
\def\HH{\mathcal{H}}

\begin{document}

\title[Curtis' connectivity theorem]{A Simple proof of Curtis' connectivity theorem \\ for Lie powers}

\author{Sergei O. Ivanov}
\address{
Laboratory of Modern Algebra and Applications,  St. Petersburg State University, 14th Line, 29b,
Saint Petersburg, 199178 Russia}\email{ivanov.s.o.1986@gmail.com}

\author{Vladislav Romanovskii}
\address{
Laboratory of Modern Algebra and Applications,  St. Petersburg State University, 14th Line, 29b,
Saint Petersburg, 199178 Russia}
\email{Romanovskiy.vladislav.00@mail.ru}

\author{Andrei Semenov}
\address{
Chebyshev Laboratory,  St. Petersburg State University, 14th Line, 29b,
Saint Petersburg, 199178 Russia}
\email{asemenov.spb.56@gmail.com}

\thanks{
The work is supported by a grant of the Government of the Russian Federation for
the state support of scientific research, agreement 14.W03.31.0030 dated 15.02.2018. The third author was also supported by ``Native Towns'', a social investment program of PJSC ``Gazprom Neft''.}

\maketitle
\begin{abstract}
We give a simple proof of the Curtis'  theorem: if $A_\bullet$ is $k$-connected free simplicial abelian group, then $L^n(A_\bullet)$ is an $k+ \lceil \log_2 n \rceil$-connected simplicial abelian group, where $L^n$ is the functor of $n$-th Lie power. In the proof we do not use Curtis' decomposition of Lie powers. Instead of this we use the Chevalley-Eilenberg complex for the free Lie algebra.
\end{abstract}

\section{\bf Introduction}

In \cite{Curtis'65}  Curtis constructed a spectral sequence that converges to the homotopy groups $\pi_*(X)$ of a simply connected space $X.$ It was described in the language of simplicial groups. This spectral sequence was an early version of the unstable Adams spectral sequence (see \cite[\S 9]{Curtis'71}, \cite{6}). Recall that a simplicial group $G_\bullet$ is called $n$-connected if $\pi_i(G_\bullet)=0$ for $i\leq n.$ For a group $G$ we denote by $\gamma_n(G)$ the $n$-th term of its lower central series. In order to prove the convergence of this spectral sequence, Curtis proved a theorem, that we call ``Curtis' connectivity theorem for   lower central series''. It can be formulated as follows. 

\medskip

\noindent {\bf Theorem} (\cite{Curtis'65}). {\it If $G_\bullet$ is a $k$-connected free simplicial group for $k\geq 0$, then the simplicial group $\gamma_n (G_\bullet)$ is $k+ \lceil \log_2 n \rceil$-connected.}

\medskip 

Curtis gave a tricky proof of this theorem using some delicate calculations with generators in free groups. Later Rector \cite{Rector} described a mod-$p$ analogue of this spectral sequence where the lower central series is replaced by the mod-$p$ lower central series. Then Quillen \cite{Quillen} found a more conceptual way to prove the connectivity theorem for the mod-$p$ lower central series using simplicial profinite groups. This result was enough to prove the convergence of the mod-$p$ version of the spectral sequence. Quillen reduced this connectivity theorem to an earlier result of Curtis, which we call ``Curtis' connectivity theorem for Lie powers''. Denote by $L^n:{\sf Ab}\to {\sf Ab}$ the functor of $n$-th Lie power.  Then the theorem can be formulated as follows. 

\medskip

\noindent {\bf Theorem} (\cite{Curtis'63}). {\it If $A_\bullet$ is a $k$-connected free simplicial abelian group, then the simplicial abelian group $L^n(A_\bullet)$ is $k+ \lceil \log_2 n \rceil$-connected.}  

\medskip

The Curtis' proof of this theorem is quite complicated and takes up most of the paper (see \cite[\S 4 -- \S 7]{Curtis'63}). He used so-called ``decomposition of Lie powers'' into smaller functors. The decomposition is a kind of filtration on the functor $L^n$ (see \cite[\S 4]{Curtis'63}).  The goal of this paper is to give a simpler proof of this theorem without the decomposition. Instead of this we use the Chevalley-Eilenberg complex for the free Lie algebra. We also generalize the statement to the case of modules over arbitrary commutative ring.     

Let $R$ be a commutative ring. We say that a functor $\FF:{\sf Mod}(R)\to {\sf Mod}(R)$ is  {\it $n$-connected} if for any $k\geq 0$ and any $k$-connected free simplicial module $A_\bullet$ the simplicial module $\FF(A_\bullet)$ is $k+n$-connected.  In these terms we prove the following.

\medskip

\noindent {\bf Theorem.} {\it The functor of Lie power $L^n:{\sf Mod}(R)\to {\sf Mod}(R)$ is $\lceil \log_2 n \rceil$-connected.} 

\medskip

We also note that this estimation of connectivity of $L^n$ is the best possible for $n=2^m.$ This is an easy corollary of the description of homotopy groups of $2$-restricted Lie powers on the language of lambda-algebra given in \cite{Curtis'71} and \cite{6} (see also \cite{Mikhailov}, \cite{Leibowitz}).   

\medskip

\noindent {\bf Proposition.} {\it If $R=\ZZ$ or $R=\ZZ/2$ the functor  $L^{2^n}:{\sf Mod}(R)\to {\sf Mod}(R)$ is not $n+1$-connected.}

\medskip 

\noindent Note that our proof of the main theorem is quite elementary. However, the proposition is a corollary of some non-elementary results about the lambda-algebra. 

Assume that $\gg$ is a Lie algebra which is free as a module over the ground commutative ring $R$. By the Chevalley-Eilenberg complex of $\gg$ we mean the chain complex whose components are exterior powers $\Lambda^i \gg$  and whose homology is homology of the Lie algebra with trivial coefficients $H_*(\gg).$ We consider the free Lie algebra as a functor from the category of free modules to the category of Lie algebras. The free Lie algebra has a natural grading whose components are Lie powers $L^*(A)=\bigoplus_{n\geq 1} L^n(A)$. Here we treat Lie powers as functors from the category of free modules $L^n:{\sf FMod}(R)\to {\sf Mod}(R).$ The grading on the free Lie algebra induces a grading on the Chevalley-Eilenberg complex whose components give exact sequences of functors on the category of free modules:
\begin{equation*}
\begin{split} 
 0 \to \Lambda^2 \to L^2 \to 0,&  \\
0\to \Lambda^3 \to {\rm Id}\otimes L^2 \to L^3 \to 0,&  \\
0 \to \Lambda^4 \to  \Lambda^2 \otimes L^2  \to ({\rm Id}\otimes L^3) \oplus \Lambda^2 L^2 \to L^4 \to 0, & \\ \dots \\
0 \to \Lambda^n \to \dots \to \bigoplus_{
\substack
{
k_1+\dots+k_n=i \\
k_1\cdot 1+k_2 \cdot 2 + \dots+k_n\cdot n=n
}
}
\ \Lambda^{k_1} L^1 \otimes \Lambda^{k_2} L^2 \otimes \dots \otimes \Lambda^{k_n} L^n \to \dots \to L^n \to 0
\end{split}
\end{equation*}
(see Corollary \ref{cor_C}), where $\Lambda^{k_s}L^s$ denotes the composition of the functor of Lie power and the functor of exterior power. We use these complexes for induction in the proof of the main result. 

\section{\bf Graded Chevalley-Eilenberg complex}

Throughout the paper  $R$ denotes a commutative ring. All algebras, modules, simplicial modules, tensor products and exterior powers are assumed to be over $R.$ 

Let $\gg$ be a Lie algebra which is free as a module. If we tensor the Chevalley-Eilenberg resolution $V_\bullet(\gg)$ (see \cite[XIII \S 7-8]{CartanEilenberg}) on the trivial module $R,$ we obtain a complex 
$C_\bullet(\gg)\cong R\otimes_{U\gg}V_\bullet(\gg)$ that we call the Chevalley-Eilenberg complex. Its components are exterior powers of the Lie algebra  $C_i(\gg)=\Lambda^i \gg$ and the differential is given by the formula
$$d(x_1\wedge \dots \wedge x_i )= \sum_{s<t} (-1)^{s+t} [x_s,x_t]\wedge x_1 \wedge \dots \wedge \hat x_s \wedge \dots \wedge \hat x_t \wedge \dots \wedge x_i.$$  
The homology of this complex is isomorphic to the homology of the Lie algebra $\gg$ with trivial coefficients
$$H_i(\gg,R)=H_i(C_\bullet(\gg)).$$

Let $\gg$ be a graded Lie algebra $\gg=\bigoplus_{n\geq 1} \gg_n.$ By a graded Lie algebra we mean a usual Lie algebra (not a Lie superalgebra) $\gg$ together with a decomposition into direct sum of modules $\gg=\bigoplus_{n\geq 1} \gg_n$ such that $[\gg_n,\gg_m]\subseteq \gg_{n+m}$ for all $n,m\geq 1.$ The degree of a homogeneous element $x\in \gg_n$ is denoted by $|x|=n.$ 

For $n\geq 1$ we consider a submodule  $C_i^{(n)}( \gg )$  of $C_i(\gg)$ spanned by elements $x_1\wedge \dots \wedge x_i,$ where $x_1,\dots,x_i$ are homogeneous and $|x_1|+\dots +|x_i|=n.$
$$C_i^{(n)}(\gg)={\rm span} \{ x_1\wedge \dots \wedge x_i \in \Lambda^i\gg\ \mid\ |x_1|+\dots +|x_i|=n \}.$$
It is easy to see that $d(C_i^{(n)}(\gg))\subseteq C_{i-1}^{(n)}(\gg),$ and hence we obtain a subcomplex $C_\bullet^{(n)}(\gg) $ of $C_\bullet(\gg).$

\begin{Proposition}\label{prop_CE}
Let $\gg=\bigoplus_{n\geq 1} \gg_n$ be a graded Lie algebra, where $\gg_n$ is free as module for each $n$. Then the Chevalley-Eilenberg complex $C_\bullet(\gg)$ has a natural grading 
$$C_\bullet(\gg)=\bigoplus_{n\geq 1} C_{\bullet}^{(n)}(\gg),$$
and there is a natural isomorphism 
$$C_i^{(n)}(\gg)\cong \bigoplus_{
\substack
{
k_1+\dots+k_n=i \\
k_1\cdot 1+k_2 \cdot 2 + \dots+k_n\cdot n=n
}
}
\ \Lambda^{k_1} \gg_1 \otimes \Lambda^{k_2} \gg_2 \otimes \dots \otimes \Lambda^{k_n} \gg_n.
$$
Here the sum runs over the set of ordered $n$-tuples of non-negative integers $(k_1,\dots,k_n)$ such that  $k_1+\dots+k_n=i$ and $k_1\cdot 1+k_2 \cdot 2 + \dots+k_n\cdot n=n.$
\end{Proposition}
\begin{proof} 
For any  modules $A,B$ there is an isomorphism $\Lambda^i(A\oplus B)\cong \bigoplus_{k+l=i} \Lambda^k(A)\otimes \Lambda^l(B).$ By induction we obtain the isomorphism 
$$\Lambda^i \left( {\bigoplus}_{s=1}^N\: A_s \right)\cong \bigoplus_{k_1+\dots +k_N=i } \Lambda^{k_1}A_1 \otimes \dots \otimes  \Lambda^{k_N} A_N.$$ 
Using the fact that the exterior power commutes with direct limits, we obtain the isomorphism for any infinite sequence of modules $A_1,A_2,\dots$
$$\Lambda^i \left( {\bigoplus}_{s=1}^\infty \: A_i \right)\cong \bigoplus_{k_1+k_2 +\dots\ =i } \Lambda^{k_1}A_1 \otimes \Lambda^{k_2} A_2 \otimes \dots.$$ 
Here we consider only sequences of non-negative integers $k_1,k_2,\dots$ where there only finitely many non-zero elements, and hence, each summand in the sum is a finite tensor product.

Take $A_n=\gg_n.$ If we have an element $x_1\wedge \dots \wedge x_i $ with homogeneous  $x_s\in \gg$ from the $R$-submodule corresponding to a summand $\Lambda^{k_1}\gg_1 \otimes \Lambda^{k_2} \gg_2 \otimes \dots,$ then $|x_1|+\dots +|x_n|=k_1\cdot 1+k_2\cdot 2+\dots.$ The assertion follows.   
\end{proof}

Let $A$ be a free module. We denote by $L^*(A)$ the free Lie algebra generated by $A$  For any basis $(a_s)$ of $A,$ $L^*(A)$ is isomorphic to the free Lie algebra generated by the family $(a_s).$ The Lie algebra $L^*(A)$ is free as a module (see \cite{Shirshov}, \cite[Cor. 0.10]{Reutenauer}). Its enveloping algebra is the tensor algebra $T^*(A)$. The map $L^*(A)\to T^*(A)$ is injective \cite[Cor. 0.3]{Reutenauer}. Hence, $L^*(A)$ can be described in terms of tensor algebra. Consider the tensor algebra $T^*(A)$ as a Lie algebra with respect to the commutator. Then $L^*(A)$ can be described as the Lie subalgebra of $T^*(A)$ generated by $A$ (see also \cite[\S 7.4]{Curtis'71}). 

The Lie algebra $L^*(A)$ has a natural grading 
$$L^*(A)=\bigoplus_{n=1}^\infty L^n(A),$$
where $L^n(A)$ is generated by $n$-fold commutators. Equivalently $L^n(A)$ can be described using the embedding into the tensor algebra as $L^n(A)=L(A)\cap T^n(A).$ The homology of $L^*(A)$  the free Lie algebra can be described as follows $H_i(L^*(A))=0$ for $i>1$ and $H_1(L^*(A))=A.$ For simplicity we set 
$${\sf C}_\bullet^{(n)}(A):=  C_\bullet^{(n)}(L^*(A)).$$
All these constructions are natural by $A.$ Denote by $L^n$ the functor of $n$-th Lie power from the category of free modules to the category of modules
$$L^n : {\sf FMod}(R)\longrightarrow {\sf Mod}(R).$$
Moreover, we treat ${\sf C}_\bullet^{(n)}$ as a complex in the category of functors ${\sf FMod}(R)\to {\sf Mod}(R).$
Then Proposition \ref{prop_CE} implies the following corollary. 

\begin{Corollary} \label{cor_C} For  $n\geq 2$ the complex ${\sf C}_\bullet^{(n)}$ of functors ${\sf FMod}(R)\to {\sf Mod}(R)$ is acyclic and has the following components 
$${\sf C}_i^{(n)}=\bigoplus_{
\substack
{
k_1+\dots+k_n=i \\
k_1\cdot 1+k_2 \cdot 2 + \dots+k_n\cdot n=n
}
}
\ \Lambda^{k_1} L^1 \otimes \Lambda^{k_2} L^2 \otimes \dots \otimes \Lambda^{k_n} L^n,$$
where $\Lambda^{k_s} L^s$ denotes the composition of the functor of Lie power and the functor of exterior power.  
Here the sum runs over the set of ordered $n$-tuples of non-negative integers $(k_1,\dots,k_n)$ such that  $k_1+\dots+k_n=i$ and $k_1\cdot 1+k_2 \cdot 2 + \dots+k_n\cdot n=n.$
\end{Corollary}

\begin{Remark} Note that  ${\sf C}_i^{(n)}=0$ for $i\notin\{1,\dots, n\},$ and that there are isomorphisms ${\sf C}_n^{(n)}=\Lambda^n$ and ${\sf C}_1^{(n)}=L^n.$ In other words ${\sf C}_\bullet^{(n)}$ is an exact sequence that connects $\Lambda^n$ and $L^n.$
$${\sf C}_\bullet^{(n)}: \hspace{1cm} 0\to \Lambda^n \to \dots \to L^n \to 0.$$
\end{Remark}

\section{\bf Connectivity of functors}

For $n\geq 0$ we say that a simplicial module $A_\bullet$ is $n$-connected, if $\pi_i(A_\bullet)=0$ for $i\leq n.$

\begin{Lemma}\label{lemma_tensor} Let $A_\bullet$ be an $n$-connected simplicial module and $B_\bullet$ be an $m$-connected free simplicial module. Then $A_\bullet\otimes B_\bullet$ is $n+m+1$-connected.  
\end{Lemma}
\begin{proof}
Consider their component-wise tensor product $A_\bullet \otimes B_\bullet.$ The Eilenberg-Zilber theorem imply that $\pi_i(A_\bullet\otimes B_\bullet)\cong H_i(NA_\bullet \otimes NB_\bullet ),$ where $NC_\bullet$ denotes the Moore complex of $C_\bullet.$ Since $N_iB_\bullet$ is a direct summand of $B_i,$ it is projective module. This gives the following variant of the K\"unneth spectral sequence: 
$$E^2_{pq}=\bigoplus_{s+t=q} {\sf Tor}^R_p(\pi_s(A_\bullet),\pi_t(B_\bullet))\Rightarrow \pi_{p+q}(A_\bullet \otimes B_\bullet).$$
If $s+t\leq n+m+1,$ then either $s<n+1$ or $t<m+1.$ Hence $E_{pq}^2=0$ for $p+q\leq n+m+1.$ Therefore, $A_\bullet\otimes B_\bullet$ is $n+m+1$-connected.  
\end{proof}

A functor from the category of modules to itself
$$\FF:{\sf Mod}(R) \longrightarrow {\sf Mod}(R)$$
is said to be $n$-connected if for any $k\geq 0$ and any $k$-connected {\it free} simplicial module $A_\bullet$ the simplicial module $\FF(A_\bullet)$ is $k+n$-connected. 

\begin{Lemma}\label{lemma_conn} Let $\FF:{\sf Mod}(R)\to {\sf Mod}(R)$ be an $n$-connected functor and $\GG:{\sf Mod}(R)\to {\sf Mod}(R)$ be $m$-connected functor. Assume that $\GG$ sends free modules to free modules. Then the composition $\FF \GG$ is $n+m$-connected and the tensor product $\FF\otimes \GG$ is $n+m+1$-connected.
\end{Lemma}
\begin{proof}
The fact about the composition is obvious.
The fact about the tensor product follows from Lemma \ref{lemma_tensor}.
\end{proof}

\begin{Lemma}\label{lemma_resolution} Let $$0 \to \FF_n\to  \dots \to  \FF_1 \to \FF_0 \to \GG\to 0 $$
be an exact sequence of functors such that $\FF_i$ is $n-i$-connected. Then $\GG$ is $n$-connected.
\end{Lemma}
\begin{proof}
The proof is by induction. For $n=0$ this is obvious. Assume that $n\geq 1$ and that the statement holds for smaller numbers. Set $\HH:={\sf Ker}(\FF_0\to \GG).$ Then by the induction hypothesis $\HH$ is $n-1$-connected. The long exact sequence 
$$\dots \to \pi_i(\HH(A_\bullet))\to \pi_i(\FF_0(A_\bullet))\to \pi_i(\GG(A_\bullet)) \to \pi_{i-1}(\HH(A_\bullet)) \to \dots$$
implies that $\GG$ is $n$-connected. 
\end{proof}

\begin{Proposition}\label{prop_exterior} The functor of exterior power $\Lambda^n$ is $n-1$-connected. 
\end{Proposition}
\begin{proof} 
The d\'ecalage formula \cite[Prop. 4.3.2.1]{Illusie} for exterior and divided powers  $\Lambda^n, \Gamma^n$   gives a homotopy equivalence for any free simplicial module $B_\bullet$
$$\Lambda^n(B_\bullet [1] )\sim \Gamma^n(B_\bullet)[n].$$
Any $0$-connected free simplicial module $A_\bullet$ is homotopy equivelent to a simplicial module of the form $B_\bullet[1],$ where $B_\bullet$ is also a free simplicial module (it follows from the same fact for non-negatively graded chain complexes). Moreover, if $A_\bullet$ is $k$-connected, we can chose $B_\bullet$ so that $B_i=0$ for $i\leq k-1.$  Hence $\pi_i(\Lambda^n(A_\bullet))=\pi_i(\Lambda^n(B_\bullet[1]))=\pi_{i-n}(\Gamma^n(B_\bullet))=0$ for $i\leq k+n-1.$
\end{proof}

\begin{Lemma}\label{lemma_log} For any two sequences of positive integer numbers $u_1,\dots,u_m$ and $v_1,\dots,v_m$ the following inequality holds 
$$\sum_{s=1}^m (u_s +  \log_2 v_s)\geq 1+ \log_2 \left(\sum_{s=1}^m u_sv_s  \right).$$
\end{Lemma}
\begin{proof}
It is easy to prove by induction that $\prod_{s=1}^m 2^{u_s}v_s\geqslant 2\sum_{s=1}^m u_sv_s.$  If we apply logarithms, we obtain the required statement.   
\end{proof}

\begin{Theorem}
The functor of Lie power $L^n$ is $\lceil \log_2 n \rceil $-connected. 
\end{Theorem}
\begin{proof}
The proof is by induction. For $n=1$ we have $L^1={\rm Id}$ and this is obvious.   Assume that $n\geq 2$ and that the statement holds for all smaller numbers. Consider the acyclic chain complex ${\sf C}^{(n)}_\bullet$ (Corollary \ref{cor_C}).   Using Lemma \ref{lemma_resolution} we obtain that it is enough to check that the functor 
$${\sf C}^{(n)}_i=\bigoplus_{
\substack
{
k_1+\dots+k_n=i \\
k_1\cdot 1+k_2 \cdot 2 + \dots+k_n\cdot n=n
}
}
\ \Lambda^{k_1} L^1 \otimes \Lambda^{k_2} L^2 \otimes \dots \otimes \Lambda^{k_n} L^n$$
is $ \lceil \log_2 n \rceil$-connected for $i\geq 2.$ It is enough to prove this for each summand. 

Fix an $n$-tuple of  $(k_1,\dots,k_n)$ such that $k_1+\dots+k_n=i\geq 2$ and $k_1\cdot 1+k_2\cdot 2 +\dots +k_n\cdot n=n.$ Note that $i\geq 2$ implies $k_n=0.$ Some of the numbers $k_j$ equal to zero. Denote by $j_1,\dots,j_m$ the indexes corresponding to non-zero numbers $k_{j_s}\ne 0.$
By Lemma \ref{lemma_conn} the functor $\Lambda^{k_j}L^j$ is $k_j-1+\lceil \log_2 j \rceil$-connected for $j<n$. Then again by Lemma \ref{lemma_conn} the tensor product $\Lambda^{k_1} L^1 \otimes \Lambda^{k_2} L^2 \otimes \dots \otimes \Lambda^{k_n} L^n$ is \hbox{$\sum_{s=1}^m \big(k_{j_s}-1+\lceil \log_2 j_s \rceil \big)+m-1$}-connected. Using Lemma \ref{lemma_log} we obtain 
$$\sum_{s=1}^m \big(k_{j_s}-1+ \log_2 j_s  \big)+m-1 = \sum_{s=1}^m \big(k_{j_s}+ \log_2 j_s  \big)-1 \geq \log_2 n.$$ The assertion follows. 
\end{proof}

\section{\bf Connectivity of $L^{2^n}$}

For $k\geq 0$ we denote by $R[k+1]$ the chain complex concentrated in $k+1$-st degree whose $k+1$-st component is equal to $R$. The Dold-Kan corresponding simplicial module is denoted by $K(R,k+1).$ Note that $K(R,k+1)$ is a $k$-connected free simplicial module. 

\begin{Proposition}\label{prop_L^2^n}
Let $R=\ZZ$ or $R=\ZZ/2$ and $n\geq 0.$ Then $L^{2^n}$ is not $n+1$-connected. Moreover, for any $k\geq 0$ $$\pi_{n+1+k}(L^{2^n}(K(R,k+1))\ne 0.$$
\end{Proposition}
\begin{proof}
(1) Let $R=\ZZ/2.$ We fix $k$ and set $V_\bullet=K(\ZZ/2,k+1).$  Denote by $L^n_{\sf res}:{\sf Vect}(\ZZ/2)\to {\sf Vect}(\ZZ/2)$ the functor of 2-restricted Lie power (see \cite[\S 2.7]{Bahturin}, \cite[\S 7.5]{Curtis'71}). The homotopy groups $\pi_{*}(L_{\sf res}^{2^n}(V_\bullet) )$ are described in terms of the lambda-algebra in \cite[Th. 8.8]{Curtis'71} (see also \cite{6} and discussion after Theorem 7.11 in \cite{Curtis'71}): 
$$\pi_{i+k+1}(L_{\sf res}^{2^n}(V_\bullet) )\cong {\it  \Lambda}^{i,n}(k+2),$$
where ${\it \Lambda}^{i,n}(k+2)$ denotes the vector sub-space of the lambda algebra ${\it \Lambda}$ with the basis given by compositions $\lambda_{i_1} \dots \lambda_{i_n},$ where $i_{s+1}\leq 2i_s,$ $i_1+\dots + i_n=i$ and $i_1\leq k+2.$ In particular, $\lambda_1^n\in {\it \Lambda}^{n,n}(k+2)\ne 0.$ Hence $$\pi_{n+1+k}(L_{\sf res}^{2^n}(V_\bullet) )\ne 0.$$

For arbitrary simplicial Lie algebra 
$\gg_\bullet$ and $t\geq 1$ we define the map $\tilde \lambda_1:\gg_t\to \gg_{t+1}$ by the formula $\tilde \lambda_1(x)=[s_0x,s_1x],$
where $s_0,s_1$ are degeneracy maps.  Denote by $\mathfrak{i}_{k+1}$ the unit of 
$(V_\bullet)_{k+1}=R.$ Then $\tilde \lambda_1^n(\mathfrak{i}_{k+1})\in (L(V_\bullet))_{n+k+1}$ is the element representing $ \lambda_1^n \in {\it \Lambda}^{n,n}(k+2)\cong \pi_{n+1+k}(L^{2^n}_{\sf res} (V_\bullet) )$ (see \cite[Prop. 8.6]{Curtis'71}). By the definition of $\tilde{\lambda}_1,$ the element $\tilde \lambda_1^n(\mathfrak{i}_{k+1} )$ lies in the unrestricted part $L^{2^n}(V_\bullet)$ of $L^{2^n}_{\sf res}(V_\bullet).$ Therefore $\tilde \lambda_1^n(\mathfrak{i}_{k+1} )$ represents a nontrivial element of $\pi_{n+1+k}(L^{2^n}(V_\bullet) )$ and hence $$\pi_{n+1+k}(L^{2^n}(V_\bullet) )\ne 0.$$

(2) Now assume that $R=\ZZ$ and set $A_\bullet=K(\ZZ,k+1).$ We denote by $L_{\ZZ/2}^{*}$ the functor of Lie power over $\ZZ/2,$ which we already discussed, and by $L_\ZZ^*$ the functor of Lie power over $\ZZ$. Then for any free abelian group $A$ we have $L^*_\ZZ(A)\otimes \ZZ/2 \cong L^*_{\ZZ/2}(A\otimes \ZZ/2).$ The universal coefficient theorem gives the following short exact sequence
$$0 \longrightarrow \pi_{i}(L_\ZZ^{2^n}(A_\bullet)) \otimes \ZZ/2 \longrightarrow \pi_{i}(L^{2^n}_{\ZZ/2}(V_\bullet)) \longrightarrow {\sf Tor}_1^\ZZ(\pi_{i-1}( L_\ZZ^{2^n}(A_\bullet) ), \ZZ/2 )\longrightarrow 0.$$
Since the functor $L_\ZZ^{2^n}$ is $n$-connected, $\pi_{n+k}(L_\ZZ^{2^n}(A_\bullet))=0.$ Therefore $$\pi_{n+1+k}(L_\ZZ^{2^n}(A_\bullet))\otimes \ZZ/2 \cong \pi_{n+1+k}(L^{2^n}_{\ZZ/2}(V_\bullet)).$$ We already proved that $\pi_{n+1+k}(L^{2^n}_{\ZZ/2}(V_\bullet) )\ne 0.$ Hence $\pi_{n+1+k}(L_\ZZ^{2^n}(A_\bullet))\ne 0.$  
\end{proof}

\begin{Remark} The Proposition \ref{prop_L^2^n} can be also deduced from results of D. Leibowitz \cite{Leibowitz} or from unpublished results of R. Mikhailov \cite{Mikhailov}, where he describes all derived functors in the sense of Dold-Puppe of Lie powers in the case $R=\ZZ$.
\end{Remark}

\end{document}